\documentclass{amsart}

\usepackage{amsmath}
\usepackage{amsthm}
\usepackage{amsfonts}
\usepackage{amssymb}
\usepackage{bbm}

\newcommand\C{{\mathbb{C}}}

\renewcommand\P{{\mathbf{P}}}
\newcommand\E{{\mathbf{E}}}

\newcommand\Var{\mathbf{Var}}
\renewcommand\Im{{\operatorname{Im}}}
\renewcommand\Re{{\operatorname{Re}}}

\newcommand\CN{{\mathcal N}}

\newcommand\BBR {{\mathbb R}}

\newcommand\BBC{{\mathbb C}}

\theoremstyle{plain}
 \newtheorem{theorem}{Theorem}

 \newtheorem{proposition}[theorem]{Proposition}
 
 \newtheorem{lemma}[theorem]{Lemma}

\theoremstyle{definition}

\begin{document}

\title[A note on the CLT for the Eigenvalue Counting Function]{A note on the Central Limit Theorem for the
Eigenvalue Counting Function of Wigner Matrices}

\author{Sandrine Dallaporta \and Van Vu}
\address{Institut de Math\'ematiques de Toulouse, UMR $5219$ du CNRS, Universit\'e de Toulouse, F-$31062$
Toulouse, France}
\email{sandrine.dallaporta@math.univ-toulouse.fr}
\address{Department of Mathematics, Rutgers, Piscataway, NJ 08854}
\email{vanvu@math.rutgers.edu}
\thanks{V. Vu is supported by research grants from AFORS and NSF}

\begin{abstract}
The purpose of this note is to establish a Central Limit Theorem for the number of eigenvalues of a Wigner
matrix in an interval. The proof relies on the correct aymptotics of the variance of the
eigenvalue counting function of GUE matrices due to Gustavsson, and its extension to large
families of Wigner matrices by means of the Tao and Vu Four Moment
Theorem and recent localization results by Erd\"{o}s, Yau and Yin.
\end{abstract}

\maketitle


\section{Introduction}

This note is concerned with the asymptotic behavior of the eigenvalue counting function,
that is the number $N_I$ of eigenvalues falling in an interval $I$, of families of Wigner matrices, when the
size of the matrix goes to infinity.
Wigner matrices are random Hermitian matrices $M_n$ of size $n$ such that, for $i<j$, the real and imaginary
parts of $(M_n)_{ij}$ are iid, with mean $0$ and variance $\frac{1}{2}$, $(M_n)_{ii}$ are iid with mean $0$
and variance $1$. An important example of Wigner matrices is the case where the entries are Gaussian, giving
rise to the so-called Gaussian Unitary Ensemble (GUE). GUE matrices will be denoted by $M_n'$. In this case,
the joint law of the eigenvalues is
known, allowing for complete descriptions of their limiting behavior both in the global and local regimes (cf.
for example \cite{AGZ}).

Denote by $\lambda_1, \dots, \lambda_n$ the real eigenvalues of the normalized Wigner matrix
$W_n=\frac{1}{\sqrt{n}}M_n$. The classical Wigner theorem states that the empirical distribution
$\frac{1}{n}\sum_{j=1}^n\delta_{\lambda_j}$
on the eigenvalues of $W_n$
converges weakly almost surely as $ n \to \infty$ to the semicircle law
$ d \rho_{sc}(x) = \frac{1}{2\pi}\sqrt{4-x^2}{\bf 1}_{[-2,2]}(x) dx$.
Consequently, for any interval $I \subset \BBR$,
\begin{equation*}
\frac{1}{n} N_I(W_n) = \frac{1}{n} \sum_{j=1}^n {\bf 1}_{\{ \lambda_j \in I \}}
      \underset{n \to \infty}{\to} \rho_{sc}(I) \quad  \textrm{almost surely.}
\end{equation*}
At the fluctuation level, it is known, due to the particular determinantal structure of the
GUE, that

\begin {theorem} [Costin-Lebowitz \cite{CL}, Soshnikov \cite{soshnikov} (see \cite{AGZ})] \label{CLS}

Let $M_n'$ be a GUE matrix. Set $W_n'=\frac{1}{\sqrt{n}}M_n'$. Let $I_n$ be an interval in $\mathbb{R}$.
If $\Var(N_{I_n}(W_n')) \underset{n \to \infty}{\to} \infty$, then
\begin{equation}
\frac{N_{I_n}(W_n')-\E[N_{I_n}(W_n')]}{\sqrt{\Var(N_{I_n}(W_n'))}} \underset{n \to \infty}{\to}
\mathcal{N}(0,1)
\end{equation}
in distribution.
\end{theorem}

In $2005$, Gustavsson \cite{Gustavsson} was able to fully describe for GUE matrices
the asymptotic
behavior of the variance of the counting function $N_I(W_n')$ for intervals $I=[y,+\infty)$ with $y \in
(-2,2)$
strictly in the bulk of the semicircle law. He  established that
\begin {equation} \label {Gustavsson}
\E[N_I(W_n')]=n\rho_{sc}(I)+O \Big (\frac{\log n}{n} \Big ) \quad {\hbox {and}} \quad
\Var(N_I(W_n'))=\Big(\frac{1}{2\pi^2}+o(1)\Big)\log n.
\end{equation}
In particular therefore, if $I=[y,+\infty)$ with $y \in (-2,2)$,
\begin {equation} \label {CLT}
\frac{N_I(W_n')-\E[N_I(W_n')]}{\sqrt{\Var(N_I(W_n'))}} \underset{n \to \infty}{\to} \CN(0,1).
\end{equation}
as well as
\begin{equation} \label {CLTnumerics}
\frac{N_I(W_n')-n\rho_{sc}(I)}{\sqrt{\frac{1}{2\pi^2}\log n}} \underset{n \to \infty}{\to} \CN(0,1)
\end{equation}
(which we call below the CLT with numerics).

The purpose of this note is to extend these conclusions to non-Gaussian Wigner matrices.
The class of Wigner matrices covered by our results is described by the following condition.
Say that $M_n$ satisfies condition $(\C0)$ if the real part $\xi$ and the imaginary part $\tilde{\xi}$ of
$(M_n)_{ij}$ are independent and have an exponential decay: there are two constants $C$ and $C'$ such that
\begin{equation*}
\P(|\xi|\ge t^C) \le e^{-t} \quad \textrm{and} \quad \P(|\tilde{\xi}|\ge t^{C}) \le e^{-t},
\end{equation*}
for all $t \ge C'$.

Say that two complex random variables $\xi$ and $\xi'$ match to order $k$ if
\begin{equation*}
\E\left[ \Re(\xi)^m \Im(\xi)^l\right]
    =\E\left[\Re(\xi')^m \Im(\xi')^l\right]
\end{equation*}
for all $m, l \ge 0$ such that $m+l \le k$.

The following theorem is the main result of this note.

\begin{theorem} \label {theorem:main}
Let $M_n$ be a random Hermitian matrix whose entries satisfy condition $(\C0)$ and
match the corresponding entries of GUE up to order $4$. Set $W_n=\frac{1}{\sqrt{n}}M_n$.
Then, for any $y \in (-2,2)$ and $I(y)=[y,+\infty)$, setting $Y_n:= N_{I(y)}(W_n)$, we have
\begin{equation*}
\E[Y_n] =n\rho_{sc}(I(y))+o(1) \quad \textrm{and} \quad \Var(Y_n) =\Big(\frac{1}{2\pi^2}+o(1)\Big)\log n, 
\end{equation*}
and the sequence $(Y_n)$ satisfies the CLT
\begin{equation*}
\frac{Y_n -\E[Y_n] }{\sqrt{\Var(Y_n)}} \underset{n \to \infty}{\to} \CN(0,1).
\end{equation*}
\end{theorem}

The theorem is established in the next two sections. In a first step, relying on Gustavsson's results and its extension to Wigner matrices by Tao and Vu \cite {Tao_Vu_wigner},
we establish that $(Y_n)$ satisfies the CLT with numerics \eqref {CLTnumerics}. In a second step,
we use recent results of Erd\"os, Yau and Yin \cite{EYY}
on the localization of eigenvalues in order to prove that
$\E[Y_n]$ and $\Var(Y_n)$ are close to those  of $M_n'$ (GUE)
and therefore satisfy  \eqref {Gustavsson}.

\section{CLT with numerics and eigenvalues in the bulk} 

On the basis of the CLT with numerics, Gustavsson \cite {Gustavsson}
described the Gaussian behavior of eigenvalues in the bulk of the semicircle law in the form of
\begin{equation}\label{behavior_eigenvalue}
\sqrt{\frac{4-t(i/n)^2}{2}} \frac{\lambda_i(W_n') - t(i/n)}{\frac{\sqrt{\log n}}{n}} \underset{n \to
\infty}{\to}
\CN(0,1)
\end{equation}
in distribution, where $t(x) \in [-2,2]$ is defined for $x \in [0,1]$ by
\begin{equation*}
x=\int_{-2}^{t(x)}d\rho_{sc}(t)=\frac{1}{2\pi}\int_{-2}^{t(x)}\sqrt{4-x^2} dx.
\end{equation*} 
More informally,
$\lambda_i(W_n') \approx t(i/n) + \CN( 0, \frac{2\log n}{(4-t(i/n)^2)n^2} )$.
This is achieved by the tight relation between eigenvalues and the counting function
expressed by the elementary equivalence, for $I(y) = [y,+ \infty)$, $y \in \mathbb{R}$,
\begin{equation}\label {magic_formula}
\ N_{I(y)}(W_n) \le n-i \quad 
\textrm{if and only if} \quad \lambda_i \le y.
\end{equation}

The result \eqref{behavior_eigenvalue} was extended in \cite{Tao_Vu_wigner} to large families
of Wigner matrices satisfying condition $(\C0)$ by means of the Four Moment
Theorem (see \cite{Tao_Vu_wigner} and \cite{Tao_Vu_wigner_edge}). Now using
the reverse strategy based on \eqref {magic_formula}, \eqref{behavior_eigenvalue}
may be shown to imply back the CLT with numerics \eqref {CLTnumerics} for
Wigner matrices whose entries match those of the GUE up to order $4$. We provide some details in
this regard relying on the following simple consequence of the
Four Moment Theorem.

\begin{proposition} \label{proposition:TV}
Let $M_n$ and $M_n''$ be two random matrices satisfying condition $(\C0)$ such that their entries match
up to
order $4$.
There exists $c>0$ such that, if $n$ is large enough, for any $y \in (-2, 2)$ and any $(i,j) \in
\{1,\dots,n\}^2$, if $I(y)=[y,+ \infty)$,
\begin{equation*}
\big |\P (\lambda_i \in I(y) )- \P(\lambda_i'' \in I(y)) \big | \le n^{-c} ,
\end{equation*}
and
\begin{equation*}
\big |\P (\lambda_i \in I(y)  \wedge \lambda_j \in I(y) )
     - \P(\lambda_i'' \in I(y)  \wedge \lambda_j'' \in I(y)) \big | \le n^{-c} .
\end{equation*}
\end{proposition}

As announced, we  would like to show that the behavior of eigenvalues in the bulk
\eqref{behavior_eigenvalue} extended to Wigner matrices leads to the CLT with numerics for such matrices,
namely, 
\begin{equation} \label {CLTWigner}
\frac{N_{I(y)}(W_n)-n\rho_{sc}(I(y))}{\sqrt{\frac{1}{2\pi^2}\log n}} \underset{n \to \infty}{\to} \CN(0,1),
\end{equation}
in distribution for Wigner matrices $W_n$ satisfying $(\C0)$. To prove this, observe that for every
$x \in \BBR$.
\begin{equation*}
\P\Big(\frac{N_{I(y)}(W_n)-n\rho_{sc}(I(y))}{\sqrt{\frac{1}{2\pi^2}\log n}} \le x\Big) 
  = \P(N_{I(y)}(W_n) \le  n-i_n)
\end{equation*}
where $i_n=n\rho_{sc}((-\infty,y])-x\sqrt{\frac{1}{2\pi^2}\log n}$. Then, by
\eqref {magic_formula},
\begin{align*}
\P\Big(\frac{N_{I(y)}(W_n)-n\rho_{sc}(I(y))}{\sqrt{\frac{1}{2\pi^2}\log n}} \le x\Big) & = 
\P(\lambda_{i_n}(W_n)
\le y)\\
& =  \P\Big(\sqrt{\frac{4-t(i_n/n)^2}{2}} \frac{\lambda_{i_n}(W_n)-t(i_n/n)}{\frac{\sqrt{\log n}}{n}} \le
x_n\Big),
\end{align*}
where $x_n=\sqrt{\frac{4-t(i_n/n)^2}{2}}\frac{y-t(i_n/n)}{\frac{\sqrt{\log n}}{n}}$.
Now $\frac{i_n}{n} \to \rho_{sc}((-\infty,y]) \in (0,1)$.
Furthermore, $x_n \to x$ since
\begin{eqnarray*}
t(i_n/n) & = & t\bigg(\rho_{sc}((-\infty,y])-\frac{x}{n}\sqrt{\frac{1}{2\pi^2}\log n}\bigg)\\
& = &
t\Big(\rho_{sc}((-\infty,y])\Big)-t'\Big(\rho_{sc}((-\infty,y])\Big)\frac{x}{n}\sqrt{\frac{1}{2\pi^2}
\log
n}+o\Big(\frac{\sqrt{\log n}}{n}\Big)\\
& = & y-x\sqrt{\frac{2}{4-y^2}}\frac{\sqrt{\log n}}{n}+o\Big(\frac{\sqrt{\log n}}{n}\Big).
\end{eqnarray*}
Hence $\frac{y-t(i_n/n)}{\frac{\sqrt{\log n}}{n}}=x\sqrt{\frac{2}{4-y^2}}+o(1)$, from which $x_n
\to x$.

Applying \eqref{behavior_eigenvalue} (extended to Wigner matrices), we obtain that  
$$\P\Big(\sqrt{\frac{4-t(i_n/n)^2}{2}} \frac{\lambda_{i_n}(W_n)-t(i_n/n)}{\frac{\sqrt{\log n}}{n}} \le
x_n\Big)
\underset{n \to \infty}{\to} \P(X \le x), $$
where $X \sim \CN(0,1)$, implying \eqref {CLTWigner}.

\section{ Estimating the mean and the variance of $Y_n$ }

To reach the CLT of Theorem \ref {theorem:main} from the CLT with numerics \eqref {CLTWigner},
it is necessary to suitably control the expectation and variance $\E[Y_n]$ and $\Var(Y_n)$
of the eigenvalue counting function, and to show that their behaviors
are identical to the ones for GUE matrices. The direct use of the Four Moment Theorem is unfortunately
not enough to this purpose since it only provides proximity on a small number of eigenvalues.
But recent results of Erd\"os, Yau and Yin \cite{EYY}, presented in the following statement,
describe strong localization of the eigenvalues of Wigner matrices which
provides the additional step necessary to complete the argument.

\begin{theorem}  \label{theorem:EYY}
Let $M_n$ be a random Hermitian matrix whose entries satisfy condition $(\C0)$.
There is a constant $C >0$ such that, for any $i \in \{1,\dots,n\}$,
\begin{equation*}
\P( |\lambda_i - t(i/n)| \ge (\log n)^{C\log \log n} \min(i,n-i+1)^{-1/3}n^{-2/3} ) \le n^{-3} .
\end{equation*}
\end{theorem} 
Note that if $n\varepsilon \le i \le (1-\varepsilon)n$ for some small $\varepsilon>0$, then
$\min(i,n-i+1) \ge n\varepsilon$ so that
\begin{equation}\label{localization}
\P \Big (|\lambda_i - t(i/n)| \ge n^{-1}\varepsilon^{-1/3}(\log n)^{C\log \log n} \Big ) \le n^{-3}.
\end{equation}

The next lemma presents the main conclusion on
the expectation and variance of the eigenvalue counting function,
extending Gustavsson's conclusion \eqref {Gustavsson} for the GUE to
Wigner matrices of the class $(\C0)$.
Once this lemma is established, Theorem \ref {theorem:main} will follow.

\begin{lemma} \label {lemma}
Set $W_n=\frac{1}{\sqrt{n}}M_n$,  $I(y)=[y,+\infty)$ where $ y \in (-2,2)$, and
$Y_n=N_{I(y)}(W_n)$. Then
\begin{equation*}
\E[Y_n]=n\rho_{sc}(I(y))+o(1) \quad \textrm{and} \quad \Var(Y_n)=\Big(\frac{1}{2\pi^2}+o(1)\Big)\log n.
\end{equation*}
\end{lemma}

\begin{proof}
By Gustavsson's results \eqref {Gustavsson} therefore, if
$Y_n'$ denotes $N_{I(y)}(W_n')$ in the case $M_n'$ is GUE,
\begin{equation*}
\E[Y_n'] =n\rho_{sc}(I(y))+O \Big (\frac{\log n}{n} \Big )  \quad \textrm{and} \quad \Var(Y_n')=
\Big(\frac{1}{2\pi^2}+o(1)\Big)\log n.
\end{equation*}
Hence, to establish Lemma \ref {lemma},  it suffices to show that 
$ \E[Y_n] = E[Y_n'] +o(1)$ and $ \Var(Y_n) = \Var(Y_n')  +o(1) $.
Below, we only deal with the variance, the argument for the expectation being
similar and actually simpler.

Set  $A_i = {\bf 1}_{\{ \lambda_i \in I \}} $, for $i \in \{1,\dots,n\}$. 
Notice that 
\begin{equation*}
|\Var (Y_n) - \Var (Y_n') | \le \sum_{1\le i,j\le n } |( \E[A_iA_j] - \E[A_i] \E[A_j] ) -( \E[A_i' A_j'] -
\E[A_i'] \E[A_j'] ) |.
\end{equation*}
Call an index $i$ {\it first class}  if $\E[A_i] \ge 1 -n^{-3}$ or $ \le n^{-3} $ and {\it second class} 
otherwise.

Note that if $j$ is first class, then, for all $i \in \{1, \dots,n\}$,
\begin{equation*}
| \E[A_iA_j] - \E[A_i] \E[A_j] | = O(n^{-3} ).
\end{equation*}
Indeed, if $\E[A_j] \le n^{-3}$, then both terms between the absolute value signs are
$O(n^{-3})$, so that $| \E[A_iA_j] - \E[A_i] \E[A_j] | = O(n^{-3} ).$ The other case can be brought back to
this case by the identity
\begin{equation*}
| \E[A_iA_j] - \E[A_i] \E[A_j] | = |\E[B_iB_j] - \E[B_i] \E[B_j] | ,
\end{equation*}
where $B_i$ is the complement of $A_i$.
Consequently,
\begin{equation}\label{class_1}
\sum_{\substack{i \ \textrm{or} \ j \\ 1^{\textrm{st}} \ \textrm{class} }}
|\big(\E[A_iA_j]-\E[A_i]\E[A_j]\big)-\big(\E[A_i'A_j']-\E[A_i']\E[A_j'] \big)|=
 O(n^{-1}).
\end{equation}

Theorem \ref{theorem:EYY} shows that there are 
only $O((\log n)^{C\log \log n})$ second class indices. Indeed, set $\eta_n=n^{-1}\varepsilon^{-1/3}(\log
n)^{C\log \log n}$ and suppose first that $i \in \{1, \dots, n\}$ is
such
that $t(i/n) < y-\eta_n$:
\begin{itemize}
\item if $t(i/n) > t(\varepsilon)$, \eqref{localization} is true for $W_n$. Then
\begin{equation*}
\P(\lambda_i \in I_n) \le \P \Big (|\lambda_i - t(i/n)| \ge \eta_n \Big )
\le n^{-3}.
\end{equation*}
\item if $t(i/n) < t(\varepsilon)$, choose $j$ such that $t(\varepsilon) < t(j/n) <
y-\eta_n$ (take $\varepsilon$ small enough and $n$ large
enough such that
there is such a $j$). Then $\lambda_i \le \lambda_j$ and
$\P(\lambda_i \in I_n)=\P(\lambda_i \ge y) \le \P(\lambda_j \ge y)=\P(\lambda_j \in
I_n) \le n^{-3}$.
\end{itemize}
Similarly one can show that if $i \in \{1, \dots, n\}$ is such that $t(i/n) >
y+\eta_n$, then $i$
is first class.

As a consequence of this discussion, $i \in \{1, \dots, n\}$
can only be second class if
$y-\eta_n<t(i/n)<y+\eta_n$.
We need to count these possible $i$'s.
By definition of $t(i/n)$, $i=\frac{n}{2\pi}\int_{-2}^{t(i/n)}\sqrt{4-x^2}dx$. Thus, 
\begin{equation*}
\frac{n}{2\pi}\int_{-2}^{y-\eta_n}\sqrt{4-x^2}dx
\le i \le
\frac{n}{2\pi}\int_{-2}^{y+\eta_n}\sqrt{4-x^2}dx.
\end{equation*}
In this case $i$ belongs to an interval of length 
\begin{equation*}
\frac{n}{2\pi}\int_{y-\eta_n}^{y+\eta_n}\sqrt{4-x^2}dx \leq  \frac{2}{\pi
\varepsilon^{1/3}}(\log n)^{C\log \log n}.
\end{equation*}

Therefore, there are at most
$\frac{2}{\pi \varepsilon^{1/3}}(\log n)^{C\log \log n}+1=O((\log n)^{C\log \log n})$ second class
$i$'s.

Next, by  Proposition \ref{proposition:TV}, it is easily seen that if both $i,j$ are second class, then
\begin{equation*}
|( \E[A_iA_j] - \E[A_i] \E[A_j] ) -( \E[A_i' A_j'] - \E[A_i'] \E[A_j'] ) |= O(n^{-c})
\end{equation*}
for some positive constant $c$. Since the number of such pairs is $O((\log n)^{2C\log \log n})$, we
have
\begin{equation}\label{class_2}
\sum_{\substack{i \ \textrm{and} \ j \\ 2^{\textrm{nd}} \
\textrm{class}}}\hspace{-7pt}
|\big(\E[A_iA_j]-\E[A_i]\E[A_j]\big)-\big(\E[A_i'A_j']-\E[A_i']\E[A_j']\big)|=
O(n^{-c}(\log n)^{2C\log \log n}).
\end{equation}

To conclude,
\begin{align*}
|\Var (Y_n) - \Var (Y_n') | & \le \hspace{-5pt} \sum_{\substack{i \ \textrm{or} \ j \\ 1^{\textrm{st}} \
\textrm{class} }} \hspace{-5pt}
|\big(\E[A_iA_j]-\E[A_i]\E[A_j]\big)-\big(\E[A_i'A_j']-\E[A_i']\E[A_j'] \big)|\\
& + \hspace{-5pt} \sum_{\substack{i \
\textrm{and} \ j \\ 2^{\textrm{nd}} \
\textrm{class}}} \hspace{-5pt} |\big(\E[A_iA_j]-\E[A_i]\E[A_j]\big)-\big(\E[A_i'A_j']-\E[A_i']\E[A_j']\big)|,
\end{align*}
so that \eqref{class_1} and \eqref{class_2} lead to
\begin{equation*}
|\Var (Y_n) - \Var (Y_n')| \le O(n^{-1})+O(n^{-c}(\log n)^{2C\log \log n})=o(1),
\end{equation*}
as claimed. This shows that $\Var (Y_n)=\Big(\frac{1}{2\pi^2}+o(1)\Big)\log n$. As mentioned
earlier, it may be shown similarly that
$\E[Y_n]=n\rho_{sc}(I(y))+o(1)$ and the proof of Lemma \ref {lemma} is thus complete.
\end{proof}

\section{About real Wigner matrices}
In this section, we briefly indicate how the preceding results for Hermitian random matrices
may be stated similarly for real Wigner symmetric matrices. To this task, we follow the same scheme
of proof, relying in particular on the corollary of Tao and Vu Four Moment Theorem (Proposition
\ref{proposition:TV}) which also holds
in the real case (cf. \cite{O_Rourke}).

Real Wigner matrices are random symmetric matrices $M_n$ of size $n$ such that, for $i<j$, $(M_n)_{ij}$ are
iid, with mean $0$ and variance $1$, $(M_n)_{ii}$ are iid with mean $0$ and variance $2$. As in the complex
case, an important example of real Wigner matrices is the case where the entries are Gaussian, giving
rise to the so-called Gaussian Orthogonal Ensemble (GOE).

The main issue is actually to establish first the conclusions for the GOE. This has been suitably developed
by O'Rourke in \cite{O_Rourke} by means of interlacing formulas (cf. \cite{FR}).

\begin{theorem}[Forrester-Rains]
The following relation holds between matrix ensembles:
\begin{equation*}
{\rm GUE}_n = {\rm even} ({\rm GOE}_n\cup {\rm GOE}_{n+1}).
\end{equation*}
\end{theorem}
This statement can be interpreted in the following way. Take two independent matrices
from the GOE, one of size $n$ and
the other of size $n+1$. If we surimperpose the $2n+1$ eigenvalues on the
real line and then take the $n$ even
ones, they have the same distribution as the eigenvalues of a $n \times n$ matrix from the GUE.

Let $I$ be an interval in $\BBR$. Let $M_n^{\BBR}$ be a GOE matrix and $M_n^{\BBC}$ be a GUE matrix.
$W_n^{\BBR}$ and $W_n^{\BBC}$ are the corresponding normalized matrices. The preceding interlacing
formula leads to
\begin{itemize}
\item $\E[N_{I}(W_n^{\BBR})]=\E[N_{I}(W_n^{\BBC})]+O(1)$
\item $\Var(N_{I}(W_n^{\BBR}))=2\Var(N_{I}(W_n^{\BBC}))+O(1)$, if $\Var(N_{I}(W_n^{\BBC})) \underset{n \to
\infty}{\to} \infty$.
\end{itemize}

Relying on this result and on the GUE case, O'Rourke proved the following theorem:
\begin{theorem}\label{GOE}
Let $M_n^{\BBR}$ be a GOE matrix. Set $W_n^{\BBR}=\frac{1}{\sqrt{n}}M_n^{\BBR}$.
Then, for any $y \in (-2,2)$ and $I(y)=[y,+\infty)$, setting $Y_n^{\BBR}:= N_{I(y)}(W_n^{\BBR})$, we have
\begin{equation*}
\E[Y_n^{\BBR}] =n\rho_{sc}(I(y))+O(1) \quad \textrm{and} \quad \Var(Y_n^{\BBR})
=\Big(\frac{1}{\pi^2}+o(1)\Big)\log n,
\end{equation*}
and the sequence $(Y_n^{\BBR})$ satisfies the CLT
\begin{equation*}
\frac{Y_n^{\BBR} -\E[Y_n^{\BBR}] }{\sqrt{\Var(Y_n^{\BBR})}} \underset{n \to \infty}{\to} \CN(0,1).
\end{equation*}
\end{theorem}

Following exactly the same scheme as for complex Wigner matrices leads to the same conclusion: Theorem
\ref{GOE} is true for Wigner symmetric matrices, provided their entries match the corresponding entries of GOE
up to order $4$ and satisfy condition ($\C0$).

The CLT for the eigenvalue counting function has been investigated as well for families of covariance
matrices. The main conclusion of this work holds similarly in this case conditioned however on the validity of
the Erd\"{o}s-Yau-Yin rigidity theorem for covariance matrices. There is strong indication that the current
approach by Erd\"{o}s, Yin and Yau for Wigner matrices will indeed yield such a result. All the other
ingredients of the proof are besides available. Indeed, Su (cf. \cite{Su}) carried out computations for
Gaussian covariance matrices and proved both the CLT and the correct asymptotics for mean and variance. Tao
and Vu in \cite{Tao_Vu_covariance} extended their Four Moment Theorem to such matrices. Arguing then as for
Wigner matrices, we could reach in the same way a CLT with numerics for suitable families of covariance
matrices.

\end{document}